\documentclass[10pt,a4paper,reqno]{amsart}
\usepackage{amsfonts,amsthm,latexsym,amsmath,amssymb,amscd,amsmath, epsf}
\usepackage{graphicx}
\usepackage{color}

\newtheorem{definition}{Definition}
\newtheorem{theorem}{Theorem}
\newtheorem{lemma}{Lemma}
\newtheorem{corollary}{Corollary}
\newtheorem{proposition}{Proposition}
\newtheorem{remark}{Remark}

\newcommand{\Int}{\textrm{int}}

\newcommand {\bC} {\mathbb {C}}
\newcommand {\bR} {\mathbb R}
\newcommand {\bZ} {\mathbb Z}

\newcommand {\Ga} {\Gamma}

\newcommand {\eps} {\varepsilon}

\newcommand{\Log}{\text{Log}}

\newcommand {\cA} {\mathcal A}
\newcommand {\cM} {\mathcal M}
\newcommand {\cH} {\mathcal H}
\newcommand {\cC} {\mathcal C}

\newcommand {\cS} {\mathcal S}
\newcommand {\rd}{\mathbb R\deg}
\newcommand{\pa}{\partial}
\newcommand {\cri} {\mathfrak {C}}
\newcommand{\LL}{\mathcal L}

\DeclareMathOperator{\conv}{conv}
\DeclareMathOperator{\im}{im}

\DeclareMathOperator{\itr}{int}
\DeclareMathOperator{\Arg}{Arg}

\renewcommand{\ge}{\geqslant}
\renewcommand{\le}{\leqslant}

\begin{document}
             \numberwithin{equation}{section}

             \title [Intersection of the contour of an amoeba with a line]
             {On the number of intersection points of the contour of an amoeba with a line}

 \author[L.~Lang]{Lionel Lang}
\address{Department of Mathematics, Stockholm University, SE-106 91
Stockholm,         Sweden}
\email {lang@math.su.se }

\author[B.~Shapiro]{Boris Shapiro}
\address{Department of Mathematics, Stockholm University, SE-106 91, Stockholm,
            Sweden}
\email{shapiro@math.su.se}

\author[E.~Shustin]{Eugenii Shustin}
\address{School of Mathematical Sciences, Tel Aviv University, Ramat Aviv, 69978 Tel Aviv, Israel}
\email {shustin@tauex.tau.ac.il}

\date{\today}
\keywords{amoeba, contour, tropical hypersurface, $\bR$-degree}
\subjclass[2010]{Primary 14P15 Secondary 14T05; 32A60}

\begin{abstract}
In this note, we investigate the maximal number of intersection points of a line with the contour of hypersurface amoebas in $\bR^n$. We define the latter number to be the $\bR$-degree of the contour. We also investigate the $\bR$-degree of related sets such as the boundary of amoebas and the amoeba of the real part of hypersurfaces defined over $\bR$. For all these objects, we provide bounds for the respective $\bR$-degrees.
\end{abstract}

\maketitle


\section{Introduction}
\label{sec:int}

Amoebas of algebraic hypersurfaces in $(\bC^\star)^n$ were introduced in 1994 in \cite{GKZ} and since then  have been one of the  central objects of study in tropical geometry.  (An accessible introduction to amoebas can be found in \cite{Vi}.)  Amoebas enjoy a number of beautiful and important properties such as special asymptotics at infinity and convexity of all connected components of the complement, to mention a few. One way to understand the geometry of amoebas goes by studying its contour. In this perspective, we introduce the following.

\begin{definition} {\rm Given a closed semi-analytic hypersurface $H\subset \bR^n$ without boundary, we define the {\it $\bR$-degree} $\rd(H)$ as the supremum  of the cardinality of $H\cap L$ taken over all lines $L\subset \bR^n$ such that $L$ intersects $H$ transversally.  (Observe that we count  points in $H\cap L$ without multiplicity)}
\end{definition}

Our aim in this note is to provide estimates for the $\bR$-degree of four closely related types of sets $H$, namely when $H$ is

--  a tropical hypersurface,

--  the boundary of the amoeba of a hypersurface $\cH\subset (\bC^\star)^n$,

-- the amoeba of the real locus of a hypersurface $\cH\subset (\bC^\star)^n$ defined over $\bR$,

--  the contour of the amoeba of a hypersurface $\cH\subset (\bC^\star)^n$.\\
In particular, we will show that $\rd(H)$ is always finite for all $H$ as above.

For a subset $H\subset \bR^n$ that is real-algebraic (respectively piecewise real-algebraic),  the $\bR$-degree satisfies  $\rd(H)\le \deg(H)$, where $\deg(H)$ is the usual degree of $H$ (respectively the sum of the degrees of the algebraic continuation of each piece of $H$). In particular, the $\bR$-degree of a real-algebraic hypersurface is always finite. More generally, if $H$ is piecewise real-analytic, then it can happen that either $\rd(H)=\infty$ or $\rd(H)<\infty$  although the degree of the analytic continuation of $H$, is always infinite.

\smallskip
We begin our investigation of the $\bR$-degree with the case of tropical hypersurfaces. Recall that for a finite set $\cM\subset \bZ^n$, a \textit{tropical polynomial} supported on $\cM$ is a convex piecewise linear function $p: \bR^n \rightarrow \bR^n$ of the form
\[ p(x)= \max_{\alpha \in \cM} \left\langle \, x \, \vert \, \alpha \,  \right\rangle + c_\alpha \]  where $c_\alpha \in \bR$. The \textit{tropical hypersurface} associated to $p$ is the set of points $x \in \bR^n$ for which $f(x)$ is equal to at least two of its \textit{tropical monomials} $\left\langle \, x \, \vert \, \alpha \,  \right\rangle + c_\alpha$. We refer to \cite{IMS} for the basic notions.  We have the following estimate.

\begin{proposition}\label{prop:trop} Let $\cM\subset \bZ^n$ be any finite set. For any tropical hypersurface  $H\subset \bR^n$  defined by a tropical polynomial supported on  $\cM$, one has
  $$\rd(H) \le \# \cM-1.$$
Moreover, there always exists a tropical hypersurface $H$ supported on $\cM$ such that $\rd(H) = \# \cM-1.$
\end{proposition}

For a finite set $\cM\subset \bZ^n$ of Laurent monomials, denote by $\vert \LL_\cM \vert$  the space of all Laurent polynomials supported on $\cM$, up to projective equivalence.
For a hypersurface  $\cH\subset (\bC^\star)^n$ given by $\{P=0\}$ where $P\in \vert \LL_\cM \vert$,  denote by $\cA_\cH\subset \bR^n$  its {\it amoeba}, i.e. the image of $\cH$ under the logarithmic map
\[
\begin{array}{rcl}
\Log \quad : \quad (\bC^\star)^n &\to& \bR^n \\ (z_1,\dots, z_n) &\mapsto & (\log|z_1|, \dots, \log|z_n|)
\end{array}.
\]
Denote by $\pa\cA_\cH\subset \cA_\cH$ the  boundary of $\cA_\cH$ and  define the {\it critical locus} $\cri_\cH \subset \cH$ to be the set of critical points of the restriction of the map $\Log$ to $\cH$. The {\it contour} $\cC\cA_\cH \subset \cA_\cH$ is the set of critical values of $\Log_{\vert \cH}$, i.e. $\cC\cA_\cH=\Log(\cri_\cH)$.  Finally, denote by $\cS\cA_\cH$ the {\it spine} of $\cA_\cH$, see \cite{PR}. We refer to Figures \ref{fig1} and \ref{fig3} and \cite{BKS} for further illustrations. More details about the spine and the contour of amoebas can be found in \cite{PT08}.

\smallskip
It is known that the critical locus $\cri_\cH$ is a real-algebraic subvariety in $(\bC^\star)^n$.
The latter follows from the description of  $\cri_\cH\subset \cH$ as the pullback of $\bR P^{n-1}$ under the \textit{logarithmic Gauss map} $\gamma_\cH: \cH \to \bC P^{n-1}$ given by
$$\gamma_\cH(z_1,\dots, z_n) = \big[ z_1\cdot \pa_{z_1}P ; \dots ; z_n \cdot \pa_{z_n}P\big]$$
where $P$ is a defining polynomial of  $\cH$, see \cite[Lemma 3]{Mi00}. Since $\cC\cA_\cH$ is the image of $\cri_\cH$ under the analytic map $\Log$, the contour $\cC\cA_\cH$ is necessarily semi-analytic, that is $\cC\cA_\cH$ is defined by analytic equations and inequalities.
 It is claimed at various places in the literature that $\cC\cA_\cH$ is actually analytic. The latter fact is not true in general as illustrated by Example 2 in \cite{Mi00}, see Section \ref{sec:proofs} for further details. Instead, we have the following.

\begin{lemma}\label{lem:analytic}
For any algebraic hypersurface $\cH\subset (\bC^\star)^n$, the contour $\cC\cA_\cH$ and the boundary $\pa\cA_\cH$ are closed semi-analytic hypersurfaces without boundary in $\bR^n$.
\end{lemma}

Let us also mention that the contour $\cC\cA_\cH$ may have components of various dimensions. Although such phenomenon has not been observed by the authors, this occurs for the critical locus of the coordinatewise argument map $\Arg$ (consider real hypersurfaces for instance). Since, in logarithmic coordinates, $\Log$ and $\Arg$ are the projection onto the real and imaginary axes respectively, there is a priori no reason why the latter phenomenon should appear only on one side of the picture.

Let us now discuss the $\bR$-degree of the boundary of hypersurface amoebas. In that perspective, observe that the spine $\cS\cA_\cH\subset \bR^n$ is a tropical hypersurface (see \cite{PT08}).

\begin{proposition}\label{prop:amoeba} Let $\cM\subset \bZ^n$ be a finite set of Laurent monomials.
For any hypersurface $\cH\subset (\bC^\star)^n$ given by $\{P=0\}$ where $P\in \vert \LL_\cM \vert$, one has
$$\rd (\pa\cA_\cH)\le 2 \cdot \rd (\cS\cA_\cH).$$
Moreover, there always exists $P\in \vert \LL_\cM \vert$ such that $\rd (\pa\cA_\cH)= 2 \cdot \rd (\cS\cA_\cH)$.
\end{proposition}

\begin{figure}[h]
\begin{center}
\includegraphics[scale=1]{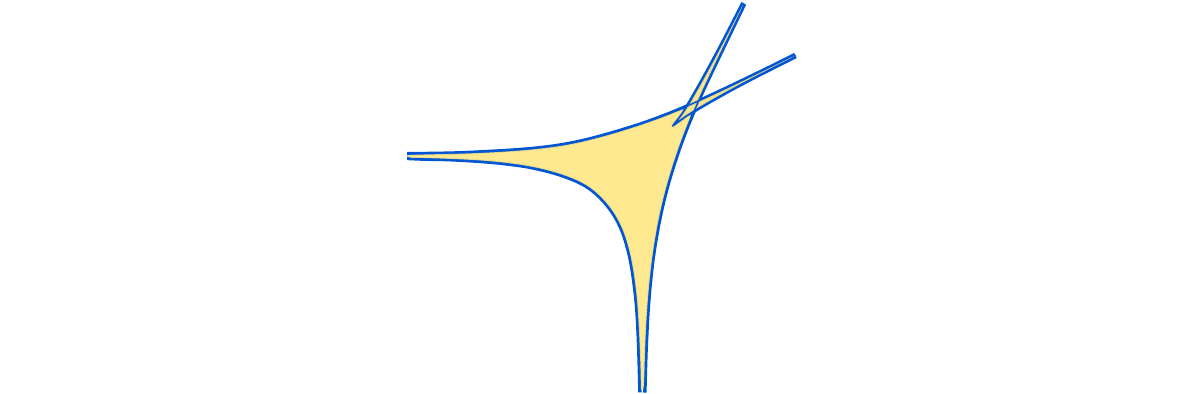}
\end{center}
\caption{Amoeba of the discriminant $27 + 4 a^3 - 18 a b - a^2 b^2 + 4 b^3$ of the family $1+ax+bx^2+x^3$ and its contour (in blue).}\label{fig1}
\end{figure}

Observe that for a particular $P\in \vert \LL_\cM \vert$,  the inequality  in Proposition \ref{prop:amoeba} can be strict, see e.g. Figure \ref{fig2}. Since the support of a  tropical polynomial defining the spine $\cS\cA_\cH$ can be always taken as a subset of $\Delta \cap \bZ^n$ where $\Delta$ is the convex hull of $\cM$ in $\bR^n= \bZ^n\otimes_{\bZ} \bR$, the following statement is a consequence of Propositions \ref{prop:trop} and \ref{prop:amoeba}.

\begin{corollary}\label{cor:bound} Under the assumptions of Proposition \ref{prop:amoeba}, one has
$$\rd (\pa\cA_\cH) \le 2\big(\#(\Delta \cap \bZ^n)-1\big)$$
where $\Delta$ is the convex hull of $\cM$.
\end{corollary}

\begin{proposition}\label{prop:solid} Let $\cM\subset \bZ^n$ be a finite set of Laurent monomials and denote $\Delta:=\conv(\cM)$. For any contractible tropical hypersurface  $T\subset \bR^n$  supported on  $\cM$, one has
  $$\rd(T) \le \# \widetilde \cM-1$$
where $\widetilde\cM:=\cM\cap \partial \Delta$. For a hypersurface $\cH\in \vert \LL_\cM \vert$ with contractible amoeba, one has
$$\rd (\pa\cA_\cH) \le 2\big(\# (\partial \Delta\cap \bZ^n)-1\big).$$
\end{proposition}

In the case of curves, we can prove a stronger statement than Corollary \ref{cor:bound}. Recall that for a non-degenerate lattice polygon $\Delta\subset\bR^2$, i.e. $\itr(\Delta)\neq \emptyset$, we can construct a toric surface $X_\Delta$ together with the tautological linear system $|\LL_\Delta|$. Denote by $V_{\Delta,g}\subset|\LL_\Delta|$ the Severi variety parametrizing
irreducible curves of genus $g$, where $0\le g\le \#(\itr(\Delta)\cap\bZ^2)$.

\begin{proposition}\label{severi}
Let $\Delta\subset\bR^2$ be a non-degenerate lattice polygon. Then, for any $0\le g\le \#(\Int(\Delta)\cap\bZ^2)$
and any curve $\cH\in V_{\Delta,g}$, one has
$$\rd(\partial\cA_\cH)\le2(\#(\partial\Delta\cap\bZ^2)-1+g).$$
Furthermore, this upper bound is sharp.
\end{proposition}

\medskip
Let us now consider the situation when $\cH \subset (\bC^\star)^n$ is a real hypersurface, i.e. its defining polynomial $P$ can be chosen to have real coefficients. Denote by
$\cH_\bR:= \cH\cap (\bR^\ast)^n$ the set of real point of $\cH$ and define the {\it real stratum} of the amoeba $\cA_\cH$ to be the set $\cA_\cH^\bR:=\Log(\cH_\bR)$. As a consequence of  \cite[Lemma 3]{Mi00}, one has the inclusions
$\cH_\bR\subset \cri_\cH$ and $\cA_\cH^\bR\subset\cC\cA_\cH$.

\begin{figure}
\begin{center}
\includegraphics[scale=1]{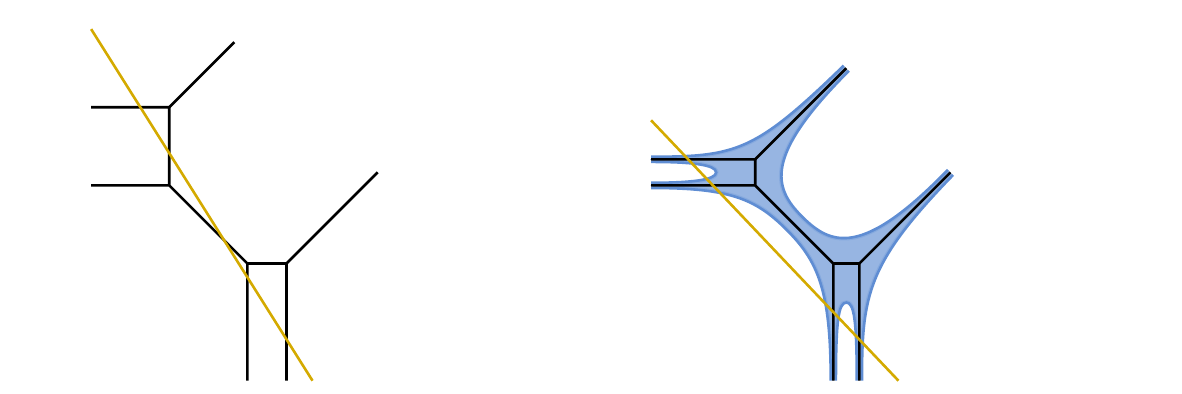}
\end{center}
\caption{
A tropical conic (left) and the amoeba of a conic (right) together with  lines realizing their $\bR$-degrees. On the right, the bounds given by Propositions~\ref{prop:amoeba} and \ref{severi} coincide and neither of them is sharp.
}
\label{fig2}
\end{figure}

Our next goal is to estimate the $\bR$-degree of $\cA_\cH^\bR$ in terms of the support $\cM$ of $\cH\subset (\bC^\star)^n$.
To formulate the answer, we need to consider the action of the group $\{\pm1\}^n$ of the sign changes of coordinates $(z_1,z_2,\dots, z_n)$ on the space
$\{\pm1\}^{\cM}$ of all possible sign patterns of the monomials in $\cM$. A sign change of coordinates $\varepsilon=(\varepsilon_1,\cdots,\varepsilon_n) \in \{\pm1\}^n$ acts on a sign pattern $\sigma=(\sigma_\alpha)_{\alpha \in \cM}$ by $\varepsilon \cdot \sigma=(\sigma_\alpha\varepsilon^\alpha)_{\alpha \in \cM}$ where $\varepsilon^\alpha=\varepsilon_1^{\alpha_1}\cdots\varepsilon_n^{\alpha_n}$ and $\alpha=(\alpha_1,\cdots, \alpha_n)$. Clearly, the cardinality of the orbit $\{\pm1\}^n\cdot\sigma$ is a power of $2$ and depends only on  $\cM$. It is therefore denoted by $2^{\kappa_\cM}$.
Notice also that either $\{\pm1\}^n\cdot\sigma=-\{\pm1\}^n\cdot\sigma$, or $\big(\{\pm1\}^n\cdot\sigma\big)\cap\big(-\{\pm1\}^n\cdot\sigma\big)=\emptyset $ and that $\{\pm1\}^n\cdot\sigma=-\{\pm1\}^n\cdot\sigma$ for some $\sigma \in \{\pm1\}^{\cM}$ if and only if it holds for all $\sigma \in \{\pm1\}^{\cM}$.

\begin{proposition}\label{pr:reallocus}  Let $\cM\subset \bZ^n$ be a finite set of Laurent monomials.
For any hypersurface $\cH\subset (\bC^\star)^n$ given by $\{P=0\}$ where $P\in \vert \LL_\cM \vert$ is a real polynomial, one has
\begin{itemize}
\item if $\{\pm1\}^n\cdot\sigma=-\{\pm1\}^n\cdot\sigma$ for all $\sigma \in \{\pm1\}^{\cM}$, then
$$\rd  (\cA_\cH^\bR)\le\begin{cases}\#\cM-1,\quad &\text{for}\ \kappa_\cM=1,\\
2^{\kappa_\cM-1}(2\#\cM-3),\quad &\text{for}\ \kappa_\cM\ge2,\end{cases}$$
\item if $\big(\{\pm1\}^n\cdot\sigma\big)\cap\big(-\{\pm1\}^n\cdot\sigma\big)=\emptyset$  for all $\sigma \in \{\pm1\}^{\cM}$, then
$$\rd  (\cA_\cH^\bR)\le\begin{cases}\#\cM-1,\quad &\text{for}\ \kappa_\cM=0,\\
2^{\kappa_\cM-1}(2\#\cM-3),\quad &\text{for}\ \kappa_\cM\ge1.\end{cases}$$
\end{itemize}
\end{proposition}

 Finally, we consider the contour of a hypersurface in $(\bC^\star)^n$.
Using Khovanskii's fewnomial theory, we obtain the following upper bound for the $\bR$-degree
of the contour.

\begin{proposition}\label{prop:contourgeneral}
 For any hypersurface $\cH\subset (\bC^\star)^n$ defined by a polynomial $P$ of degree $d$,
 one has
 $$\rd (\cC\cA_\cH) \le 2^{2n+(n-1)(n-2)/2} d^{n+1}\Big(4dn+2(n-1)^2-1\Big)^{n-1}.$$
\end{proposition}

\medskip
The upper bound of the above proposition is probably not sharp, as illustrated by the following improvement in dimension $2$
 in which case we 
 take into account the combinatorics of the Newton polygon of the curve.

\begin{proposition}\label{prop:contour}
 For any  curve $\cH\subset (\bC^\star)^2$
 defined by a bivariate polynomial $P$ of degree $d$ and with Newton polygon $\Delta$, one has
 $$\rd (\cC\cA_\cH) \le 4d^3(4d-2)+ \#(\partial \Delta \cap \bZ^2) -\text{Area}(\Delta)$$
where $\text{Area}(\Delta)$ is twice the Euclidean area of $\Delta$.
\end{proposition}

\begin{figure}
\begin{center}
\includegraphics[scale=0.15]{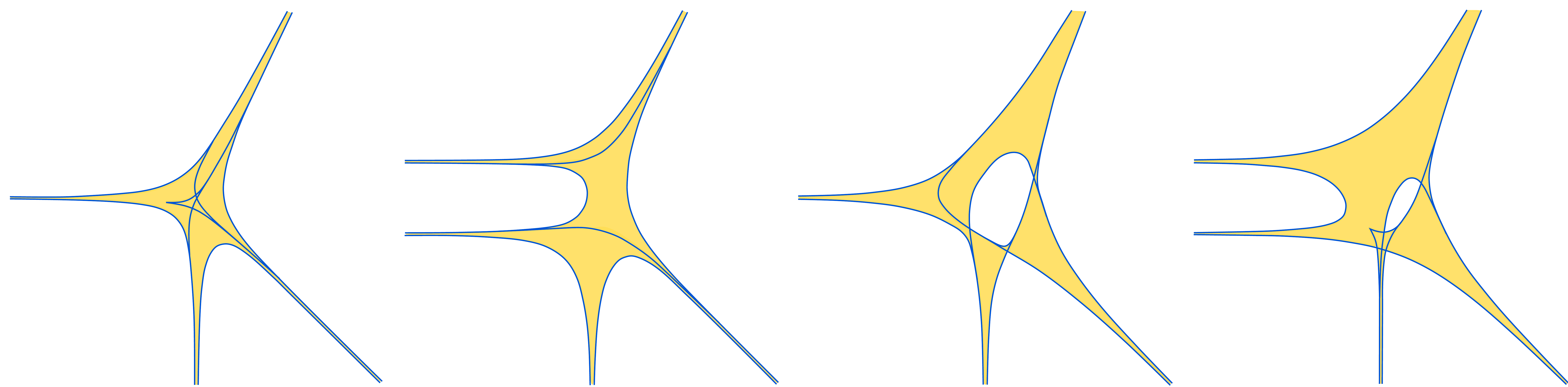}
\end{center}
\caption{Amoebas of curves with Newton polygon $\conv\{(0,0),(1,0),(2,1),(0,2)\}$ with their respective contour (in blue).}\label{fig3}
\end{figure}

While proving Proposition \ref{prop:contour}, we additionally provide an upper bound on the number of cusps of the contour $\cC\cA_\cH$, see Corollary \ref{cor:boundcr}. The latter quantity is intimately related to the analogue of Hilbert's sixteenth problem for amoebas considered in \cite{La}.

\smallskip
To conclude the introduction, let us  mention that the subject of this paper is a particular instance of the general problem of finding estimates for the number of real solutions to systems of (semi)-analytic equations. The most well-known example is the fewnomial theory developed by A.~Khovanskii in  \cite{Kh} where the considered systems of equations are given by (semi)-Pfaffian functions.

Being the image of the (real algebraic) critical locus of a complex hypersurface under the logarithmic map, the contour of an amoeba is the zero set of a sub-Pfaffian function. (Amoebas' boundary is also defined by a sub-Pfaffian function).

We think that the existing methods of obtaining upper bounds in the fewnomial  theory, mainly based on the so-called Rolle-Khovanskii lemma are not very effective for amoebas. Indeed, most of the equations of the Pfaffian system defining the critical locus $\cri_\cH$ depend only  the single polynomial equation $P$ defining the hypersurface $\cH\subset (\bC^\star)^n$. In particular, the latter system is highly non-generic. The upper bound from Proposition \ref{prop:contour}  comes from convexity of the components of the complement to an amoeba and other topological considerations which are different type of phenomena as compared to the Rolle-Khovanskii type of observations.

\smallskip
The structure of the paper is as follows. Section~\ref{sec:proofs} contains the proofs of the above statements and Section~\ref{sec:final} contains some discussions and further outlook.

\medskip
\noindent
{\bf Acknowledgements.}  The first and the third authors want to thank the Mittag-Leffler institute for the hospitality in Spring 2018. The second author wants to acknowledge the financial support of his research provided by the Swedish Research Council grant  2016-04416. The second author is sincerely grateful to D.~Novikov  and T.~Sadykov  for discussions and their interest in this project. The authors thank D.~Bogdanov for providing an online tool for automated generation of MATLAB code  available for free public use at http://dvbogdanov.ru/?page=amoeba and his help with creating Fig.~\ref{fig3}.

\section{Proofs} \label{sec:proofs}

We begin this section with a general remark that we will use repeatedly.

\begin{remark}\label{rem:semicont}{\rm
For any continuous family of lines $L_t$ intersecting $H$ transversally, the number of points $\#(L_t\cap H)$ is a lower semi-continuous function in $t$. In particular, whenever $\rd(H)$ is finite, one can always find a line $L$ with rational slope which is transversal to $H$ and such that $\# (L\cap H)=\rd(H)$. Similarly, for any continuous family of hypersurface $H_t$ intersecting $L$ transversally, the number $\#(L\cap H_t)$ is a lower semi-continuous function in $t$}.
\end{remark}

\begin{proof}[Proof of Lemma~\ref{lem:analytic}]
Let us show that the contour $\cC\cA_\cH$ is a real-analytic hypersurface in $\bR^n$. Recall that $\cC\cA_\cH$ is the image of the critical locus $\cri_\cH\subset \cH$ under the map $\Log$. As a consequence of \cite[Lemma 3]{Mi00}, the locus $\cri_\cH$ is a closed real-algebraic subvariety in $(\bC^\star)^n$. Since the map $\Log$ is real-analytic and proper, it follows that $\cC\cA_\cH$ is a closed semi-analytic subvariety in $\bR^n$, i.e. $\cC\cA_\cH$ is defined by real-analytic equalities and inequalities. Thus, it remains to prove that $\cC\cA_\cH$ has no boundary. Reasoning by contradiction, let us assume that the boundary of $\cC\cA_\cH$ is non-empty. Then, we can find a point $p \in \cri_\cH$ and two open neighborhood $U\subset \cri_\cH$ and $V\subset \bR P^{n-1}$ such that $p\in U$, $\gamma_\cH(p)\in V$ and $\gamma_\cH(U)=V$. According to \cite[Theorems 9.6.1 and 9.6.2]{BCR}, we can choose $p$, $U$ and $V$ and suitable real-analytic coordinates on $\bR^n$ such that $\Log(U)=\bR_{\ge 0}^k \times \{0\}^{n-k}$ for some $1\le k\le n-1$. In particular, we can find an $(n-k+1)$-plane $\Pi \subset\bR^n$ with rational slope (in the original coordinates) such that $\Pi \cap\Log(U)=\bR_{\ge 0} \times \{0\}^{n-k}$. Up to restricting $\cH$ to the unique $(n-k+1)$-dimensional affine subgroup of $(\bC^\star)^n$ passing through $p$ and mapping to $\Pi$, we can assume that $k=1$. In particular, the image of $\Log(U)$ under the tangent map valued in the  Grassmannian of lines in $\bR^n$ is an arc $\alpha$ with a terminal point. Now, observe that for any point $q \in \cri_\cH$ and any tangent vector $v \in T_q \cri_\cH$ such that $T_q \Log(v) \neq 0$, we have that $\gamma_\cH(q)$ lies in the hyperplane dual to  $T_q \Log(v)$. In particular, we have that $\gamma_\cH(U)$ is contained in the union of hyperplanes $ \cup_{a\in \alpha} a^\vee$ intersected with $V$. The latter is a strict subset of $V$. This is a contradiction with the fact that $\gamma_\cH(U)=V$. It follows that $\cC\cA_\cH$ has no boundary.

By definition, the set $\pa \cA_\cH$ is the boundary of the semi-analytic set $\cA_\cH$ and is therefore semi-analytic. There are two cases: either $\cA_\cH$ has empty interior or not. In the first case, the hypersurface $\cH$ is necessarily an affine subgroup of codimension $1$ of $(\bC^\star)^n$ and $\cA_\cH$ is a hyperplane in $\bR^n$. In particular, the boundary $\pa \cA_\cH$ is empty. In the second case, the boundary $\pa \cA_\cH$ has to separate the interior $\itr(\cA_\cH)$ from the complement $\bR^n\setminus \cA_\cH$. Therefore, it cannot have boundary.
\end{proof}

\begin{remark}\label{rem2}
In general, the contour of a hypersurface amoeba is not analytic. To see this, consider as in \cite[Example 2]{Mi00} the hyperbola $\cH \subset (\bC^\star)^2$ defined by $$P(z,w)=w-(z^2-2z+a)$$ where $a>1$. The latter curve is parametrized by $z\mapsto(z,z^2-2z+a)$, and the composition of the logarithmic Gauss map with the latter parametrization is given by $\gamma_\cH(z)=\big[-2(z^2-z); z^2-2z+a\big]$. Write $z=c+id$. An elementary computation shows that 
\[\gamma_\cH(z)\in\bR P^1 \quad \Leftrightarrow  \quad d=0 \; \text{ or } \; (c-a)^2+d^2=a(a-1).\]
Consequently, the critical locus $\cri_\cH \subset \cH$ consists of two components: the real part of $\cH$ (when $d=0$) and a circle of radius $a(a-1)$ intersecting the latter component in two points. At each such point, the map $\gamma_\cH$ has local normal form $\tilde z\mapsto \tilde z^2$ and the two components of $\cri_\cH$ are given respectively by $\tilde d=0$ and $ \tilde c=0$, where $\tilde z=\tilde c+i \tilde d$. In the coordinate $\tilde z$, the restriction of $\Log$ to $\cH$ is given by
\[ \Log(\tilde z)= \Big(\alpha\cdot\tilde c - \beta \cdot\tilde d^{\,2} + \text{ h.o.t }, \frac{\alpha}{3}\cdot(\tilde c^{\, 3}-3\cdot\tilde c\tilde d^{\,2})+ \frac{\beta}{2}\cdot\tilde d^{\,4} + \text{ h.o.t } \Big). \] 
In particular, the image under $\Log$ of the branch $\tilde d=0$ of the $\cri_\cH$ is analytic whereas the image of $\tilde c=0$ is only semi-analytic. Indeed, we have $$\Log(\tilde c)= \big(\alpha\tilde c + \dots, (\alpha/3)\tilde c^{\, 3}+ \dots \big) \text{ and }\Log(\tilde d) = \big(-\beta \tilde d^{\,2} + \dots, (\beta/2)\tilde d^{\,4} + \dots \big).$$
It follows that the contour $\cC \cA_\cH\subset \bR^2$ is semi-analytic but not analytic.
\end{remark}

\begin{proof}[Proof of Proposition~\ref{prop:trop}]
For any tropical hypersurface $H\subset \bR^n$ supported on $\cM$, the $\bR$-degree $\rd(H)$  is finite since $H$ is contained in the union of finitely many hyperplanes. In particular, the integer $\rd(H)$ is given as the number of the intersection points of $H$ with some line $L\subset\bR^n$ with rational slope, see Remark \ref{rem:semicont}. In other words, $\rd(H)$ is the number of tropical roots of the univariate tropical polynomial $p_L$ obtained by restricting the tropical polynomial defining $H$ to $L$. Obviously, the tropical polynomial $p_L$ is the sum of at most $\# \cM$ tropical monomials. Therefore, $p_L$ has at most $\# \cM-1$ tropical roots.

To prove that $\# \cM-1$ is a sharp upper bound for a given support set $\cM\subset \bZ^n$, notice that the direction of $L$ can be chosen so that $p_L$ has exactly $\# \cM$ monomials and that the coefficients of the tropical polynomial defining $H$ can be chosen so that $p_L$ has the maximal number of tropical roots, that is $\# \cM-1$.
\end{proof}

\begin{proof}[Proof of Proposition~\ref{prop:amoeba}]
Recall that all connected components of the complement to the amoeba $\cA_\cH\subset \bR^n$ of the hypersurface $\cH\subset (\bC^*)^n$
are always convex, see \cite[Theorem 1.1]{FPT}. Moreover, the spine $\cS\cA_\cH$ is a deformation retract of the amoeba $\cA_\cH$, see \cite[Theorem 1]{PR}. Therefore, the inclusion of the connected components of  $\bR^2 \setminus\cA_\cH$ in the connected components of $\bR^2\setminus\cS\cA_\cH$ is a $1$-to-$1$ correspondence. Now, the intersection of any line $L \subset \bR^n$ with $\cA_\cH$ is a union of intervals and we claim that each such interval $I$ intersects  $\cS\cA_\cH$  at least once. Indeed, by convexity of the connected components of $\bR^2 \setminus\cA_\cH$,  the endpoints of $I$ lie on the boundary of two different connected components of $\bR^2 \setminus\cA_\cH$. According to the above correspondence, the endpoints of $I$ necessarily belong to different connected components of $\bR^2 \setminus\cS\cA_\cH$. It implies that $I$ meets $\cS\cA_\cH$ as least once and the claim follows. Therefore, one has that $\rd (\pa\cA_\cH)\le 2\cdot \rd (\cS\cA_\cH).$

For the second part of the statement, for any given finite set $\cM\subset \bZ^n$, one can find  an amoeba which is arbitrarily close to its spine using Viro polynomials, see \cite[Corollary 6.4]{Mi04}. In particular, any line $L$ realizing $\rd(\cS\cA_\cH)$, that is such that $\#( L \cap \cA_\cH)=\rd(\cS\cA_\cH)$, has the property that  $\#(L \cap \pa\cA_\cH) = 2 \cdot \rd(\cS\cA_\cH)$.
\end{proof}

\begin{proof}[Proof of Proposition~\ref{prop:solid}] Let $P$ be a tropical polynomial defining $T$ and supported on  $\cM$.  Consider the order map sending each connected component of $\bR^n \setminus T$ to the exponent of the tropical monomial of $P$ dominating the other monomials on that component.  Observe that the latter order map sends connected components of $\bR^n\setminus T$ injectively to the set of points in $\cM$. Moreover, the unbounded connected components of this complement are sent to the set of points in $\cM$ lying on the boundary of $\Delta$, i.e. to the set $\widetilde \cM$. By convexity, the number of intersection points of any generic line $L$ with $T$ equals the number of connected components of $L \setminus T$ minus $1$.  Following the same line of arguments as in the proof of Proposition \ref{prop:trop}, we show that the latter bound is sharp for any set $\cM\subset \bZ^n$.

For the second part of the statement, observe that if the amoeba $\cA_\cH$ is contractible, i.e. there are no bounded connected components in $\bR^2\setminus\cA_\cH$, then the same holds for  the complement to the spine $\cS\cA_\cH$. It implies that the support of the spine is a subset of $\partial \Delta\cap \bZ^n$.
The result now follows from Proposition \ref{prop:trop}. \end{proof}

\begin{proof}[Proof of Proposition \ref{severi}] We can assume without any loss of generality that $\cH$ is immersed, see Remark \ref{rem:semicont}.
Reasoning by contradiction, assume that there exists a line $L\subset\bR^2$ intersecting the boundary of the amoeba $\partial\cA_\cH$ transversally in
$2n$ points, where $n\ge\#(\partial\Delta\cap\bZ^2)+g$. It follows that the threefold $\Log^{-1}(L)$ intersects the curve
$\cH$ along at least $n$ disjoint ovals. Since $\cH\cap(\bC^*)^2$ is an immersed surface of genus $g$ with at most $\#(\partial\Delta\cap\bZ^2)$ punctures, the complement $(\cH\cap(\bC^*)^2)\setminus \Log^{-1}(L)$ must contain a connected component
without punctures. In particular, the image of the latter component under $\Log$ must be bounded. In such case, the harmonic function
$a\log|x_1|+b\log|x_2|$, where $(a,b)$ is the normal vector of the line $L$, must have an extremum in the interior of the above bounded component. This is in contradiction with the maximum principle. We conclude that $n\leq \#(\partial\Delta\cap\bZ^2)-1+g$ and the result follows.   
To prove the sharpness, we provide the following example. It follows from \cite[Theorem 4 and Section 8.5]{Mi05} that given a generic line $L\subset\bR^2$ and $\#(\partial\Delta\cap\bZ^2)-1+g$ distinct points on $L$, there exists a trivalent tropical curve of genus $g$ with Newton polygon $\Delta$, which intersects $L$ in the chosen points. Furthermore, it follows from \cite[Lemma 8.3]{Mi05}
(see also \cite[Section 1]{Vi1} and \cite[Corollary 6.4]{Mi04}) that there exists an
algebraic curve $C$ of genus $g$ with Newton polygon $\Delta$, whose complex amoeba is located in
an $\eps$-neighborhood of the above tropical curve ($0<\eps\ll1$) and, on the other hand, cover a (smaller) neighborhood of that tropical curve. Thus, we encounter at least $2(\#(\partial\Delta\cap\bZ^2)-1+g)$
points in $L\cap\partial\cA_\cH$. 
\end{proof}

\begin{proof}[Proof of Proposition~\ref{pr:reallocus}]
According to Remark \ref{rem:semicont}, it suffices to consider the intersection of $\cA_\cH^\bR\subset \bR^n$ with lines $L\subset \bR^n$ with rational slope in order to calculate $\rd (\cA_\cH^\bR)$. Let $L$ be  a line parameterized by
\begin{equation}x_1=\ell_1\tau+k_1,\; x_2=\ell_2\tau+k_2,\;\dots, \;x_n= \ell_n \tau + k_n\ ,
\label{e1}\end{equation} where $\ell_1,\dots, \ell_n\in \bZ$ are coprime and  $ k_1,\dots, k_n, \tau\in\bR$. To count points in $\cA_\cH^\bR\cap L$, we need to count all points in the intersections of the real locus $\cH_\bR$   with $2^n$ real rational curves $\Ga(\eps)$ given by  $\nu \mapsto (z_1=\eps_1e^{k_1}\nu^{\ell_1},z_2=\eps_2e^{k_2} \nu^{\ell_2},\dots, z_n= \eps_n e^{k_n}\nu^{k_n})$
and restricted to the interval $\nu>0$, where $\eps:=(\eps_1,...,\eps_n)\in(\bZ_2)^n$  is an arbitrary sequence of signs.
Substituting different parameterizations $\Ga(\eps)$ in the polynomial $P$ defining $\cH$, we obtain
$2^{\kappa_{\cM}}$ different real univariate fewnomials whose positive roots correspond to the intersection points in $\cA_\cH^\bR\cap L$. According to Descartes' rule of signs \cite[Proposition 1.2.14]{BCR}, this yields the upper bound $\#(L\cap\cA_\cH^\bR) \le 2^{\kappa_{\cM}}(\#\cM-1)$.

Suppose now that $\big(\{\pm1\}^n\cdot\sigma\big)\cap\big(-\{\pm1\}^n\cdot\sigma\big)=\emptyset$ for all $\sigma \in \{\pm1\}^{\cM}$ and $\kappa_\cM\ge1$. Then there is an element
$\eps\in\{\pm1\}^n$ acting non-trivially on any orbit $\{\pm1\}^n\cdot\sigma$, splitting the latter into $2^{\kappa_\cM-1}$ disjoint pairs. The polynomials interchanged by $\eps$ are related by the substitution
$-\nu$ for $\nu$. According to Remark \ref{rem:semicont}, we can assume that the
 parametrization (\ref{e1}) is such that $\ell_i$ is odd if and only if $\eps_i=-1$ in the sequence of signs $\eps$. By assumption, the set of indices $i$ such that $\eps_i=-1$ is neither empty nor $\{1,\dots,n\}$. By Lemma \ref{l3} below, the real roots of each such polynomial have at most $2\#\cM-3$  distinct absolute values. The second bound in Proposition \ref{pr:reallocus} follows.

Similarly, when $\{\pm1\}^n\cdot\sigma=-\{\pm1\}^n\cdot\sigma$ for all $\sigma \in \{\pm1\}^{\cM}$, it is enough to notice that one half of the polynomials in
the orbit $\{\pm1\}^n\cdot\sigma$ is obtained from the other half by multiplying them by  $-1$.
\end{proof}

\begin{lemma}\label{l3}
Given an arbitrary univariate $d$-nomial, the total number of
distinct absolute values of its non-vanishing real zeros is at most $d-1$ if all the exponents are either even, or odd, and is at most $2d-3$ otherwise. Both bounds are sharp.
\end{lemma}

\begin{proof} To start with, notice that the total number of non-vanishing real roots of an arbitrary $d$-nomial is at most $2d-2$, see e.g. \cite[Proposition 1.2.14]{BCR}.
We have to show that $2d-2$ distinct absolute values of non-vanishing real roots of a $d$-nomial are impossible. Indeed, in such a case one must have alternating signs of the coefficients of both the original polynomial $P(x)$  and for the polynomial $P(-x)$. Necessarily, the polynomial $P$ is the product of a monomial and a polynomial in the square of the variable. Hence, the roots of $P$ have at most
$d-1$ distinct absolute values. The upper bound $2d-3$ is achieved e.g.  for the $d$-nomial
$x\prod_{i=1}^{d-2}(x^2-i)+t$, where $0<t\ll1$.
\end{proof}

In order to prove Proposition \ref{prop:contourgeneral},  let us recall the definition of Pfaffian manifold given in \cite[p. 5 and 6]{Kh}.

\begin{definition}
A submanifold $\Gamma\subset\bR^n$ of codimension $q$ is a \textbf{simple Pfaffian submanifold} of $\bR^n$ if there exists an ordered collection $\alpha_1, \dots , \alpha_q$ of $1$-forms on $\bR^n$ with polynomial coefficients and a chain of submanifolds $\bR^n\supset \Gamma_1\supset \dots \supset \Gamma_q=\Gamma$ such that $\Gamma_i$ is a separating solution of the Pfaff equation $\alpha_j=0$ on the manifold $\Gamma^{i+1}$. \\
Recall that a submanifold of codimension $1$ in a manifold $M$ is a \textbf{separating solution} of the Pfaffian equation $\alpha=0$ (for a $1$-form $\alpha$ on $M$) if

-- the restriction of $\alpha$ to the submanifold is identically zero,

-- the submanifold does not pass through the zeroes of $\alpha$,

-- the submanifold is the boundary of some region in $M$, and the co-orientation of the submanifold determined by the form is equal to the co-orientation of the boundary of the region.
\end{definition}

The theorem below is given in \cite[p. 6]{Kh}. Its proof can be found in  \cite[$\S 3.12$]{Kh}.

\begin{theorem}\label{thm:kho}
The number of non-degenerate roots of a system of polynomial equations $P_1=\dots=P_k=0$ on a simple Pfaffian submanifold in $\bR^n$ of dimension $k$ is bounded from above by
\[ 2^{q(q-1)/2} p_1\dots p_k \Big( \sum(p_j-1)+mq-1\Big)^q\]
where $q=n-k$, the polynomial $P_j$ has degree $p_j$ and the coefficients of the forms defining the Pfaffian submanifold have degree bounded by $m$.
\end{theorem}

\begin{proof}[Proof of Proposition \ref{prop:contourgeneral}]
Let $\cH \subset (\bC^\star)^n$ be an algebraic hypersurface defined by a polynomial $P$. Let $L\subset \bR^n$ be a line such that $L$ intersects $\cC\cA_\cH$ transversally and such that $\rd(\cC\cA_\cH)=\#(\cC\cA_\cH\cap L)$.  Our first goal is to show that $\Log^{-1}(L)\subset (\bC^\star)^n$ is a simple Pfaffian manifold.

Let $(\ell_1,\dots,\ell_n)\in\bR^n$ be a supporting vector for $L$. If $J\subset \{1,\dots,n\}$ is the subset of indices $j$ for which $\ell_j\neq 0$, then $\rd(\cri_\cH)=\rd(\cri_\cH\cap(\bC^\star)^J)$. Therefore, up to intersecting with $(\bC^\star)^J$, we can assume with no loss of generality that $\ell_j\neq 0$ for all $j$. The case $n=1$ is trivial so let us assume that $n\ge 2$. There exists a vector $(\eps_1,\dots,\eps_n)\in (\bR^\star)^n$ such that the parametrised curve $\rho: \nu \mapsto (\eps_1 \nu^{\ell_1},\dots,\eps_n \nu^{\ell_n})\subset (\bC^\star)^n$, $\nu\in \bR^\star$, is such that $\im(\Log\circ \rho)=L$. For technical reasons, we will need to ensure that not all the $\ell_j$ have the same sign. If they share the same sign, consider the change of coordinates $(z_1,z_2,\dots,z_n) \mapsto (z_1^{-1},z_2,\dots,z_n)$ on $(\bC^\star)^n$ so that $\ell_1$ is replaced with $-\ell_1$. Up to a change of the  polynomial $P$ defining $\cH$ by $z_1^d\cdot P$, we can now assume that not all the $\ell_j$ have the same sign and that the polynomial $P$ has degree at most $2d$. Consider now the partial compactification $(\bC^\star)^n \subset \bC^n$. Since there exist $j\neq k$ such that  $\ell_j$ and $\ell_k$ have different signs, the  subset $\im(\rho)$ is equal to its closure in $(\bC^\star)^n$. In other words,  $\overline{\im(\rho)}$ does not intersect any of the coordinate axes of $\bC^n$. Since $\Log^{-1}(L)=\Log^{-1}\big(\im(\Log\circ\rho)\big)$, we deduce that $\Log^{-1}(L)$ is also disjoint from the coordinate axes. We can now describe $\Log^{-1}(L)$ as a simple Pfaffian submanifold of $\bC^n=\bR^{2n}$. Indeed, consider $n-1$ linear equations
\[
\left\lbrace
\begin{array}{l}
b_{1,1}t_1+\dots+b_{1,n}t_n=c_1,\\
\hspace{1.4cm} \dots \\
\hspace{1.5cm} \dots \\
b_{n-1,1}t_1+\dots+b_{n-1,n}t_n=c_{n-1},
\end{array}
\right.
\]
defining the line $L\subset \bR^n$. Define the $1$-form $\beta_j:=b_{j,1} \cdot d t_1+\dots+b_{j,n} \cdot d t_n$, $1\le j \le n-1$, such that each such form is identically zero on $L$. Set $z_j:=x_j+i\cdot y_j$. Thus, each form
\[ \displaystyle
\begin{array}{ll}
\alpha_j:=  	& \displaystyle  \Big(\prod_{1\le j\le n} (x_j^2+y_j^2) \Big) \cdot \big( b_{j,1} \cdot d \ln\vert z_1\vert +\dots+b_{j,n} \cdot d \ln\vert z_n\vert \big) \\
& \\
					& \displaystyle \Big( \prod_{1\le j\le n} (x_j^2+y_j^2)\Big) \cdot \Big( b_{j,1} \cdot \frac{x_1\cdot dx_1+y_1\cdot dy_1}{x_1^2+y_1^2}+\dots+b_{j,n} \cdot  \frac{x_n\cdot dx_n+y_n\cdot dy_n}{x_n^2+y_n^2} \Big)
\end{array}
\]
is identically zero on $\Log^{-1}(L)$. Observe that each form $\alpha_j$ has polynomial coefficients of degree $2n-1$ and the zero locus of  $\alpha_j$ is exactly  the union of the coordinate axes in $\bC^n$. Since the linearly independent forms $\alpha_j$, $1\le j \le n-1$, with polynomial coefficients vanish on the analytic subset $\Log^{-1}(L)\subset \bR^{2n}$ of codimension $n-1$ and that $\Log^{-1}(L)$ avoids the  zero locus of the $\alpha_j$, it follows that $\Log^{-1}(L)$ is a simple Pfaffian submanifold of $\bR^{2n}$.

To conclude, observe that $\#(\cC\cA_\cH\cap L)\le \#\big(\cri_\cH \cap \Log^{-1}(L)\big)$ and that  $\cri_\cH$ is defined by the $n+1$ polynomial equations in the real coordinates $\bR^{2n}=\bC^n$
\[
\left\lbrace
\begin{array}{l}
\Re(P)=0,\\
\Im(P)=0,\\
\Im\big(z_1\cdot\partial_{z_1}P\cdot \overline{z_n}\cdot\overline{\partial_{z_n}P} \big)=0,\\
\hspace{2cm} \dots \\
\Im\big(z_{n-1}\cdot\partial_{z_{n-1}}P\cdot \overline{z_n}\cdot\overline{\partial_{z_n}P} \big)=0,
\end{array}
\right.
\]
where the first two equations determine $\cH$ and the remaining $n-1$ equations determine $\gamma_\cH^{-1}(\bR P^{n-1})$. The first two equations have the same degree as $P$ which is at most $2d$. The remaining equations have degree at most $4d$. According to Theorem \ref{thm:kho}, we have that $\#\big(\cri_\cH \cap \Log^{-1}(L)\big)$ is bounded from above by
\[
\begin{array}{l}
2^{(n-1)(n-2)/2} (2d)^2(4d)^{n-1} \Big( 2(2d-1)+(n-1)(4d-1)+(n-1)(2n-1)+1\Big)^{n-1}\\
\\
= 2^{2n+(n-1)(n-2)/2} d^{n+1}\Big(4dn+2(n-1)^2-1\Big)^{n-1}.
\end{array}
\]
The result follows.
\end{proof}

\medskip

In order to prove Proposition~\ref{prop:contour}, let us provide some additional information about the contour and the critical locus of curves in the linear system $\vert \mathcal{L}_{\Delta} \vert$. By Remark \ref{rem:semicont}, we can restrict our attention to any open dense subset of $\vert \mathcal{L}_{\Delta} \vert$ while proving Proposition~\ref{prop:contour}. We will therefore consider smooth curves $\cH\subset (\bC^\star)^2$ in $\vert \mathcal{L}_{\Delta} \vert$ whose critical locus $\cri_\cH$ is smooth. We assume additionally that the curve $\cH$ is \textit{torically non-degenerate}, that is $\#(\overline{\cH}\setminus \cH)=\#(\pa\Delta\cap\bZ^2)$ where $\cH$ is the closure of $\cH$ in the compactification $X_\Delta$. Below, we refer to the above assumptions with $(\star)$.

The set of curves satisfying the above assumptions $(\star)$ is an  open dense subset of $\vert \mathcal{L}_{\Delta} \vert$. Indeed, requiring the smoothness of $\cri_\cH$ is an open dense condition according to \cite[Theorem 1]{La}. The curve $\cH$ is torically non-degenerate for a generic choice of coefficients on $\pa\Delta\cap\bZ^2$.

Denote by $\Sigma_\cH\subset \cri_\cH\subset \cH$ the set of points of $\cri_\cH$ where $\Log$ is not an immersion. Notice that the image of any point in $\Sigma_\cH$ under $\Log$ is a cusp. Here, by a {\it cusp}, we mean a germ of a plane curve parametrized as $(t,0)\to (t^p f(t), t^q g(t))$, where $p$ and $q$ are coprime positive integers exceeding $2$ and $f(t)$ and $g(t)$ are two converging power series  with $f(0)\neq 0$ and $g(0)\neq 0$. For technical reason, we assume that $\Delta$ lies in the positive quadrant, touching both coordinates axes. We denote by $P$ a polynomial with Newton polygon $\Delta$ defining the curve $\cH$ and  denote $d:=\deg(P)$.

\begin{lemma}\label{lem:rdegc} Let $\cH\subset (\bC^\star)^2$  be a curve satisfying $(\star)$. Then, the $\bR$-degree of the contour $\cC\cA_\cH$ satisfies the inequality
 $$\rd (\cC\cA_\cH) \le \#(\partial \Delta \cap \bZ^2) +\text{Area}(\Delta)+2 \vert \Sigma_\cH \vert.$$
\end{lemma}

\begin{proof}
Let $L=\{ax+by=c\}\subset \bR^2$ be a line realizing the $\bR$-degree of the contour $\cC\cA_\cH$ and consider the pencil $L_\kappa:=\{ax+by=c+\kappa\},\; \kappa\in \bR$, of lines parallel to $L$. Observe that when $|\kappa|>>0$, then $\cC\cA_\cH \cap L_\kappa$ consists exactly of the intersection of $L_\kappa$ with the boundary of the tentacles of the amoeba $\cA_\cH$ whose supporting ray in the normal fan of $\Delta$ sit in the half-plane $\{ax+by>0\}$. Let $b_+$ be the number of primitive integer segments on $\pa \Delta$ whose outer normal vector sits in $\{ax+by>0\}$ and define $b_-$ by the relation $\#(\partial \Delta \cap \bZ^2)=b_++b_-$. Then for $|\kappa|>>0$, the line $L_\kappa$ intersects exactly $b_+$ tentacles of $\cA_\cH$ (by toric non degeneracy), implying that $\#(\cC\cA_\cH \cap L_\kappa) = 2\, b_+$. When $\kappa$ decreases back to $0$, the number of intersection points of $\cC\cA_\cH$ with $L_\kappa$
changes  either when $L_\kappa$ becomes tangent to a branch of $\cC\cA_\cH$ or when $L_\kappa$ passes through a point $\Log(p)$ with $p\in \Sigma_\cH$. The points in $\bR^2$ where the tangency occurs are exactly the point in $\Log\big(\gamma_\cH^{-1}([a;b])\big)$. The latter set decomposes into the disjoint union of two subsets of points for which $\# (\cC\cA_\cH \cap L_\kappa )$ changes respectively by $-2$ and $+2$ when $\kappa$ decreases (by genericity of $L$, we can assume that $\gamma_\cH^{-1}([a;b])\cap \Sigma_\cH=\emptyset$). Denote by $\gamma_-$ and $\gamma_+$ the cardinality of the respective subsets. While passing through a cusp of $\cC\cA_\cH$, we also have that $\#(\cC\cA_\cH \cap L_\kappa)$ might change by $\pm2$ with a priori  no control on the sign of the contribution. We deduce that
\[\rd (\cC\cA_\cH)=\#(\cC\cA_\cH \cap L) \leq 2\big(b_++\gamma_+ +\vert \cri_\cH\vert \big).\]
To conclude, it remains to estimate $2(b_++\gamma_+)$. To do this, observe  that $$\gamma_++\gamma_-=\deg \gamma_C= \text{Area}(\Delta) \;  \text{ and } \; 2b_++2\gamma_+-2\gamma_-=2b_-.$$
The first equality comes from \cite[Lemma 2]{Mi00} whereas the second one comes from counting the number of ovals in $\Log^{-1}(L_\kappa)\cap \cH$ when $\kappa$ decreases from $+\infty$ to $-\infty$.
If we denote by $g$ the number of the inner lattice points in $\Delta$, we deduce from Pick's formula that
\[
\begin{array}{l}
\left\lbrace
\begin{array}{l}
\gamma_++\gamma_-=b_++b_-+2g-2\\
2b_++2\gamma_+-2\gamma_-=2b_-
\end{array}
\right.
\Leftrightarrow
\left\lbrace
\begin{array}{l}
\gamma_++\gamma_-=b_++b_-+2g-2\\
\gamma_+-\gamma_-=b_--b_+
\end{array}
\right.
\Leftrightarrow
\\
\\
\left\lbrace
\begin{array}{l}
\gamma_++\gamma_-=b_++b_-+2g-2\\
2\gamma_+=2b_++2g-2
\end{array}
\right.
\Rightarrow
2(b_++\gamma_+)=2(b_++b_-)+2g-2
\Rightarrow
\\
\\
2(b_++\gamma_+)= \#(\partial\Delta\cap\bZ^2)+\text{Area}(\Delta)\leq \#(\partial\Delta\cap\bZ^2)+\text{Area}(\Delta).
\end{array}
\]
The result follows.
\end{proof}

In order to bound the $\bR$-degree of $\cC\cA_\cH$ and to settle Proposition \ref{prop:contour}, it remains to estimate from above the number of points in $\Sigma_\cH$. In order to characterize the points of $\Sigma_\cH$, we need to introduce the $2$-plane field $K(p): = \text{Ker}\big(T_p \Log\big)$, $p\in (\bC^\star)^2$, where $T\, \Log$ is the tangent map for the map $\Log:
(\bC^\star)^2\to \bR^2$. One can check that at any point $p=(z,w)\in (\bC^\star)^2$, one has $K(p):=\bR iz\oplus \bR iw$. Equivalently, the $2$-plane $K(p)$ is  orthogonal to the $2$-plane $\bR \cdot (z,0) \oplus \bR \cdot (0,w)$ with respect to the standard scalar product on $\bC^2=(\bR\oplus i\bR)^2=\bR^4$. Indeed, we have $\left\langle(z_1,w_1) \vert (z_2,w_2) \right\rangle = \Re (z_1\overline{w_1}+z_2\overline{w_2})$.

Given a smooth plane curve $\cH\subset (\bC^\star)^2$, observe that its critical locus
$\cri_\cH\subset \cH$ is characterized by the property that for each $p\in \cri_\cH$, the tangent line $T_p\cH$ to $\cH$ at $p$ is non-transversal to $K(p)$. In this case their intersection is necessarily $1$-dimensional since the plane $K(p)$  is never a complex line. Let us denote the latter intersection line by $\mathcal \ell_p$, $p\in  \cri_\cH$. The next lemma is obvious.

\begin{lemma}\label{lem:cusps} In the above notation, a smooth point  $p\in\cri_\cH$ belongs to $\Sigma_\cH$ if and only if the line $\mathcal \ell_p$ is tangent to $\cri_\cH$  at $p$.
\end{lemma}

\begin{proposition}\label{prop:eqs}
For a smooth curve $\cH\subset (\bC^\star)^2$  satisfying $(\star)$,
the set $\Sigma_\cH\subset (\bC^\star)^2$  is given by (the real solutions of) the following overdetermined system of $8$ real algebraic equations in $4$ variables:
\begin{equation}
\begin{cases}
\Re \big(P(z,w)\big)=0\\
\Im \big(P(z,w)\big)=0\\
\Im \big(z \cdot \partial_z P(z,w)\cdot \overline{w} \cdot \overline{\partial_w P(z,w)}\big)=0\\
\text{rank } M(z,w) \le 3
\end{cases}\label{eq:sys}
\end{equation}
where $M(z,w)$ is the  $5\times 4$-matrix with  rows given by
$$\begin{pmatrix}
\text{grad } \Re (P)\\
\text{grad } \Im (P)\\
\text{grad } \Im \big(z\, \partial_zP\, \overline{w} \, \overline{\partial_w P}\big)\\
(z,0)\\
(0,w)
\end{pmatrix}.
$$
\end{proposition}

\begin{proof}
Denote by $M_i(z,w)$ the $4\times 4$-submatrix obtained from $M(z,w)$ by removing the $i^{th}$ row. Then, the first three equations of  \eqref{eq:sys} determine the critical locus $\cri_\cH$. Indeed, the first two equations define the vanishing locus of $P$ and the third equation corresponds to $\gamma_\cH^{-1}(\bR P^1)$.  The condition $\text{rank } M(z,w) \le 3$ is equivalent to the vanishing of the $5$ maximal minors of $M$ which gives five extra equations in addition to  the first three equations of \eqref{eq:sys}. Observe that if a point $p=(z,w)$ satisfies the first three equations of  \eqref{eq:sys}, we have the following:

-- the first two rows of $M(z,w)$ are linearly independent since they describe the tangent space $T_{p} \cH$  to the smooth curve $\cH$;

-- the first three rows of $M(z,w)$ are linearly independent since they describe the tangent line to $\cri_\cH$ which is smooth by assumption;

-- the last two rows of $M(z,w)$ are linearly independent since they describe the $2$-plane $K(p)$;

-- the minor $\det \big( M_3(z,w)\big)$ vanishes since $K(p)$ and $T_p \cH$ intersect along a line.

Assume that $p=(z,w)$ satisfies the first three equations of  \eqref{eq:sys}. From the above observations, it follows that $\text{rank } M(z,w) \le 3$ if and only if the kernel of the last two rows of $M(z,w)$ contains the kernel of the first three rows of $M(z,w)$. Equivalently, we have that $\text{rank } M(z,w) \le 3$ if and only if $K(p)$ contains the line tangent to $\cri_\cH$ at $p$. Since the latter line lies in the tangent space $T_p \cH$ and that $K(p)\cap T_p \cH=\ell_p$ is $1$-dimensional, this is equivalent to $\ell_p$ being tangent to $\cri_\cH$ at $p$. Finally, by Lemma \ref{lem:cusps}, this is equivalent to $p\in \Sigma_\cH$.
\end{proof}

\begin{remark} {\rm From the above proof, we deduce that the set $\Sigma_\cH\subset (\bC^\star)^2$ is the set of all real solutions of the following overdetermined system of $5$ equations in $4$ variables:
\begin{equation}
\begin{cases}
\Re \big(P(z,w)\big)=0\\
\Im \big(P(z,w)\big)=0\\
\Im \big(z \cdot \partial_z P(z,w)\cdot \overline{w} \cdot \overline{\partial_w P(z,w)}\big)=0\\
\det M_{4}(z,w)=0\\
\det M_{5}(z,w)=0
\end{cases}.\label{eq:sys2}
\end{equation}
}\label{rem:eqs}
\end{remark}

Thus the number of points of $\Sigma_\cH$ is bounded from above by the number of real solutions of any  sub-system of \eqref{eq:sys2} containing $4$ equations. In turn, the number of real solutions of such a  system is bounded by the number of its complex solutions that can be estimated either by using Bernstein-Kouchnirenko's Theorem or, more roughly, using B\'{e}zout's  Theorem.

If we remove the penultimate or the last equation of \eqref{eq:sys2}, we can explicitly identify the real solutions of the corresponding square system that are not in $\Sigma_\cH$. To that aim, we need to assume that the image of $\Sigma_\cH$ under the logarithmic Gauss map is disjoint from $\big\lbrace [0;1] , \, [1;0] \big\rbrace \subset \bC P^{1}$. The latter assumption is always satisfied after an appropriately chosen toric change of coordinates in $(\bC^\star)^2$.

\begin{proposition}\label{prop:lionel} Assume that  $\cH\subset (\bC^\star)^2$ satisfies $(\star)$ and that $\gamma_\cH \big(\Sigma_\cH \big)$ is disjoint from $\big\lbrace [0;1] , \, [1;0] \big\rbrace \subset \bC P^{1}$. Then, the set of real solutions of the square system obtained from \eqref{eq:sys2} by removing the penultimate (respectively the last) equation is the disjoint union of $\Sigma_\cH$ with $\gamma_\cH^{-1}\big([1;0] \big)$ $\big($respectively $\gamma_\cH^{-1}\big([0;1] \big) \,\big)$.
\end{proposition}

\begin{proof}
Let $p=(z,w)$ be a point satisfying the first three equations of \eqref{eq:sys2}, that is $p\in \cri_\cH$. The first three rows of $M(z,w)$ shared by $M_4(z,w)$ and $M_5(z,w)$ define the tangent line to $\cri_\cH$ at $p$ in the tangent space $T_p\cH$. The latter line is spanned by a vector of the form $\big(-e^{i\theta} \cdot \partial_w P(z,w), \, e^{i\theta} \cdot \partial_z P(z,w) \big)$ where  $\theta \in \bR$ is unique up to $\pi\bZ$. According to Remark \ref{rem:eqs}, the point $p$ belongs to  $\Sigma_\cH$ if and only if
\[\Re \big(e^{i\theta}  \cdot\partial_w P(z,w)  \cdot\overline{z}\big)= \Re \big(e^{i\theta} \cdot \partial_z P(z,w)  \cdot \overline{w}\big)=0.\]
By the assumption on $\gamma_\cH$, both $\partial_w P(z,w) $ and $\partial_z P(z,w)$ are different from $0$ for $p\in \Sigma_\cH$. Assume now that $p \in \cri_\cH$ only satisfies the first four equations of \eqref{eq:sys2}, i.e.  $\Re \big(e^{i\theta} \cdot \partial_z P(z,w)  \cdot \overline{w}\big)=0$. To start with, observe that  all the points in $\gamma_\cH^{-1}\big([0;1] \big)$ satisfy the latter equation since $\partial_z P(z,w)=0$ for such points. The second observation is that $\partial_w P(z,w)\neq 0$, otherwise $p$ would belong to $\gamma_C^{-1}\big([1;0] \big) \cap \Sigma_\cH$ which is empty by assumption. Assume now that $p \notin \gamma_\cH^{-1}\big([0;1] \big)$. Then, we have $\gamma_\cH(p):=\big[z\cdot\partial_z P(z,w); w\cdot \partial_w P(z,w)]=[u;v]$ with $u\cdot v\neq 0$. Therefore,
\[
\begin{array}{l}
\quad \; \; \Re \big(e^{i\theta}\cdot \partial_z P(z,w) \cdot\overline{w}\big)=0 \; \Leftrightarrow  \; \vert w \vert ^2 \cdot \Re \big(e^{i\theta}\cdot \partial_z P(z,w) \cdot w^{-1}\big)=0 \\
\\
\Leftrightarrow  \; \Re \big(e^{i\theta}\cdot \partial_w P(z,w) \cdot z^{-1}\cdot u \cdot v^{-1}\big) =0 \; \Leftrightarrow \; u \cdot v^{-1} \cdot \Re \big(e^{i\theta}\cdot \partial_w P(z,w) \cdot z^{-1}\big)=0 \\
\\
\Leftrightarrow  \; \Re \big(e^{i\theta}\cdot \partial_w P(z,w) \cdot\overline{z}\big)=0.
\end{array}
\]
We conclude that the set of points satisfying the first four equations of \eqref{eq:sys2} is the disjoint union of $\gamma_\cH^{-1}\big([0;1] \big)$ and $\Sigma_\cH$. The proof for the square system obtained by removing the penultimate equation of \eqref{eq:sys2} is similar.
\end{proof}

\begin{corollary}\label{cor:boundcr} Under the hypotheses $(\star)$, the cardinality of $\Sigma_\cH$ does not exceed $$2d^3(4d-2) - \text{Area}(\Delta).$$
\end{corollary}

\begin{proof}
The cardinality of $\Sigma_\cH$ does not exceed the number of real solution of the square system given by the first four equations of \eqref{eq:sys2}. The first two equations of the latter system have degree $d$ and the third one has degree $2d$. For the fourth equation, each coefficient of the first two rows of $M_4(z,w)$ is a polynomial of degree $d-1$. Each coefficient of the third row of is a polynomial of degree $2d-1$ and each coefficient of the last row is linear. Thus, the equation $\det M_4(z,w)=0$ has degree $2(d-1)+(2d-1)+1=4d-2$. By B\'{e}zout's Theorem, the square system has at most $2d^3(4d-2)$ solutions. By Proposition \ref{prop:lionel}, exactly $\text{Area}(\Delta)$ of the real solutions of the square system are not in $\Sigma_\cH$.
\end{proof}

\begin{proof}[Proof of Proposition~\ref{prop:contour}]
The statement follows from Lemma \ref{lem:rdegc} and Corollary \ref{cor:boundcr}.
\end{proof}

\section{Final remarks}\label{sec:final}

Here we formulate a few questions related to the $\bR$-degree of the contour of amoebas.

\smallskip
\noindent
{\bf 1.} Propositions ~\ref{prop:contourgeneral} and ~\ref{prop:contour} give upper bounds for the $\bR$-degree of the contour of hypersuface amoebas. However, these bounds are apparently not sharp.  Do they  have a correct order of magnitude in term of the degree of the hypersurface?

\smallskip
\noindent
{\bf 2.} Can we generalize the geometric approach of Proposition~\ref{prop:contour} to arbitrary dimension in order to improve the bound of Propositions ~\ref{prop:contourgeneral}?

\smallskip
\noindent
{\bf 3.}  As we mentioned in the introduction the contour $\cC\cA_\cH$ is in general only semi-analytic, but it seems  (as claimed in the earlier literature) that typically it will be analytic.  It would be interesting to formulate some simple sufficient non-degeneracy condition guaranteeing the validity of the latter property.

\smallskip
\noindent
{\bf 4.} There seems to exist a certain ``compensation rule" for the boundary of an amoeba and the rest of its contour meaning that if the complement to the boundary of an amoeba has a  simple topology (as for example,  in the case of a contractible amoeba), then the rest of the contour has many singularities and a  complicated topology. Reciprocally,  the boundary of the amoeba of a Harnack curve has the maximal possible number of ovals and it coincides with the whole contour, see Fig.~\ref{fig3}. (This figure is borrowed from  \cite[Section~3]{BKS}) with the kind permission of the authors;  in loc. cit. one can  find the explicit forms of the polynomials whose amoebas are shown in Fig.~\ref{fig3}  as well as  some further discussions.)

The final challenge is to make a quantitative statement describing the above experimental observation.


\begin{thebibliography}{30}




\bibitem[BCR]{BCR}  Bochnak, J.,  Coste, M.,   Roy, M-F., Real algebraic geometry. Springer- Verlag, 1998.

\bibitem[BKS]{BKS} Bogdanov, D., Kutmanov, A., and Sadykov, T., {\em Algorithmic computation of polynomial amoebas}, CASC 2016: Computer Algebra in Scientific Computing pp 87--100.


\bibitem[FPT]{FPT} Forsberg, M.,  Passare, M., Tsikh, A., {\em Laurent determinants and arrangements of hyperplane amoebas}, Advances in Math. {\bf 151}  (2000), 45--70.

\bibitem[GKZ]{GKZ}  Gelfand, I. M.,   Kapranov, M. M.,  Zelevinsky, A. V., Discriminants, resultants, and multidimensional determinants. Mathematics: Theory \& Applications. Birkh\"auser Boston, Inc., Boston, MA, 1994. x+523 pp.


%


\bibitem[IMS]{IMS} Itenberg, I.; Mikhalkin, G.; Shustin, E.;
Tropical algebraic geometry.
Second edition. Oberwolfach Seminars, 35. Birkh\"{a}user Verlag, Basel, 2009. x+104 pp.

\bibitem[Kh]{Kh} Khovanskii, A.~G., Fewnomials, Translations of Mathematical Monographs 88 (AMS, Prov-
idence, RI, 1991).

\bibitem[La]{La} Lang, L., {\em Amoebas of curves and the Lyashko-Looijenga map}, J. London Math. Soc., doi:10.1112/jlms.12214

\bibitem[Mi00]{Mi00} Mikhalkin, G., {\em Real algebraic curves, the moment map and amoebas}, Ann. of Math. (2), {\hbox {\bf 151:1}}, (2000), 309--326.
    
\bibitem[Mi04]{Mi04}  Mikhalkin, G.,
{\em Decomposition into pairs-of-pants for complex algebraic hypersurfaces},
Topology, {\bf 43:5},
(2004), 1035--1065.

\bibitem[Mi05]{Mi05} Mikhalkin, G.,  {\em
Enumerative tropical algebraic geometry in $\bR^2$},
J. Amer. Math. Soc. {\hbox{\bf 18}} (2005), 313--377.



\bibitem [PR]{PR} Passare, M., Rullg\aa{}rd, H., {\em  Amoebas, Monge--Amp\`ere measures and
triangulations of the Newton polytope},
 Duke Math. J., {\hbox{\bf 121} } (2004), 481--507.



\bibitem[PT08]{PT08} Passare, M., Tsikh, A., {\em Amoebas: their spines and their contours,} Idempotent mathematics and mathematical physics, Contemp. Math. {\bf 377}, (2005), 275--288.

\bibitem[Vi01]{Vi1} Viro, O., {\em  Dequantization of real algebraic geometry on logarithmic paper}. European Congress of Mathematics, Vol. I (Barcelona, 2000), Progr. Math., 201, Birkh\"auser, Basel, 2001, pp. 135--146.

\bibitem[Vi02]{Vi} Viro, O., {\em What is an amoeba?}  Notices of the AMS, {\hbox{\bf49}}, No. 8, September 2002, 916--917.


\end{thebibliography}
\end{document}